\documentclass{conm-p-l}
\usepackage{amssymb, amsthm, amsmath, amscd}

\newtheorem{theorem}{Theorem}[section]
\newtheorem{proposition}[theorem]{Proposition}
\newtheorem{corollary}[theorem]{Corollary}

\theoremstyle{remark}
\newtheorem{remark}[theorem]{Remark}
\newtheorem{example}[theorem]{Example}

\numberwithin{equation}{theorem}

\def\Proj{\operatorname{Proj}}

\def\Cl{\operatorname{Cl}}
\def\div{\operatorname{div}}
\def\Div{\operatorname{Div}}
\def\gcd{\operatorname{gcd}}
\def\lcm{\operatorname{lcm}}

\def\coker{\operatorname{coker}}
\def\image{\operatorname{image}}
\def\to{\longrightarrow}

\def\ge{\geqslant}
\def\le{\leqslant}

\def\lf{\lfloor}
\def\rf{\rfloor}

\def\bx{\mathbf x}
\def\frakm{\mathfrak m}
\def\frakp{\mathfrak p}

\def\CC{\mathbb C}
\def\NN{\mathbb N}
\def\PP{\mathbb P}
\def\QQ{\mathbb Q}
\def\ZZ{\mathbb Z}

\def\calO{\mathcal O}

\begin{document}
\title{Divisor class groups of graded hypersurfaces}

\author{Anurag K. Singh}
\address{Department of Mathematics, University of Utah, 155 South 1400 East, Salt Lake City, UT~84112, USA} \email{singh@math.utah.edu}
\author{Sandra Spiroff}
\address{Mathematics Department, Seattle University, 901 12th Ave, Seattle, WA~98122, USA} \email{spiroffs@seattleu.edu}

\thanks{The first author was supported by the NSF under grants DMS~0300600 and
DMS~0600819.}


\begin{abstract}
We demonstrate how some classical computations of divisor class groups can be obtained using the theory of rational coefficient Weil divisors and related results of Watanabe.
\end{abstract}
\maketitle

\section{Introduction}

The purpose of this note is to provide a simple technique to compute divisor class groups of affine normal hypersurfaces of the form 
\[
k[z,x_1,\dots,x_d]/(z^n-g)\,,
\]
where $g$ is a weighted homogeneous polynomial in $x_1,\dots,x_d$ of degree relatively prime to $n$. We use the theory of rational coefficient Weil divisors due to Demazure \cite{Demazure} and related results of Watanabe \cite{Watanabe}. This provides an alternative approach to various classical examples found in Samuel's influential lecture notes \cite{Samuel}, as well as to computations due to Lang \cite{LangTams} and Scheja and Storch \cite{SchejaStorch}. While the computations we present here are subsumed by those of \cite{SchejaStorch}, our techniques are different. A key point in our approach is that the projective variety defined by a hypersurface as above is weighted projective space over $k$, and this makes for straightforward, elementary calculations.

Watanabe \cite[page~206]{Watanabe} pointed out that $\QQ$-divisor techniques can be used to recover the classification of graded factorial domains of dimension two, originally due to Mori \cite{Mori}. Robbiano has applied similar methods to a study of factorial and almost factorial schemes in weighted projective space \cite{Robbiano}.

\section{$\QQ$-divisors}

We review some material from \cite{Demazure} and \cite{Watanabe}. Let $k$ be a field, and let $X$ be a normal irreducible projective variety over $k$, with rational function field $k(X)$.

A \emph{rational coefficient Weil divisor} or a \emph{$\QQ$-divisor} on $X$ is a $\QQ$-linear combination of irreducible subvarieties of $X$ of codimension one. Let $D=\sum n_iV_i$ be a $\QQ$-divisor, where $V_i$ are distinct. Then $\lf D\rf$ is defined as
\[
\lf D\rf =\sum \lf n_i\rf V_i\,,
\]
where $\lf n\rf$ denotes the greatest integer less than or equal to $n$. We set
\[
\calO_X(D)=\calO_X(\lf D\rf)\,.
\]
If each coefficient $n_i$ occurring in $D$ is nonnegative, we say that $D\ge 0$.

A $\QQ$-divisor $D$ is \emph{ample} if $nD$ is an ample Cartier divisor for some $n\in\NN$. In this case, the \emph{generalized section ring} corresponding to $D$ is the ring
\[
R(X,D)=\oplus_{j\ge0}H^0(X,\calO_X(jD))\,.
\]
If $R=R(X,D)$, then the $n$-th Veronese subring of $R=R(X,D)$ is the ring
\[
R^{(n)}=\oplus_{j\ge0}H^0(X,\calO_X(jnD))=R(X,nD)\,.
\]
The following theorem, due to Demazure, implies that a normal $\NN$-graded ring $R$ is determined by a $\QQ$-divisor on $\Proj R$.

\begin{theorem}\label{thm:Demazure} \cite[3.5]{Demazure}.
Let $R$ be an $\NN$-graded normal domain, finitely generated over a field $R_0$. Let $T$ be a homogeneous element of degree $1$ in the fraction field of $R$. Then there exists a unique ample $\QQ$-divisor $D$ on $X=\Proj R$ such that
\[
R=\oplus_{j\ge0}H^0(X,\calO_X(jD))T^j\,.
\]
\end{theorem}

We next recall a result of Watanabe, which expresses the divisor class group of $R$ in terms of the divisor class group of $X$ and a $\QQ$-divisor corresponding to $R$.

\begin{theorem}\label{exactseq} \cite[Theorem~1.6]{Watanabe}
Let $X$ be a normal irreducible projective variety over a field. Assume $\dim X\ge 1$ and let $D=\sum_{i=1}^r(p_i/q_i)V_i$ be a $\QQ$-divisor on $X$ where $V_i$ are distinct irreducible subvarieties, $p_i,q_i\in\ZZ$ are relatively prime, and $q_i>0$. Set
\[
R=\oplus_{j\ge0}H^0(X,\calO_X(jD))T^j\,.
\]
Then there is an exact sequence
\[\CD
0@>>>\ZZ@>{\theta}>>\Cl(X)@>>>\Cl(R)@>>>\coker\alpha@>>>0\,,
\endCD\]
where $\theta(1)=\lcm(q_i)\cdot D$, and $\alpha\colon\ZZ\to\oplus_{i=1}^r\ZZ/q_i\ZZ$ is the map $1\mapsto (p_i\mod q_i)_i$.
\end{theorem}

In the exact sequence above, $\coker\alpha$ is always a finite group. Moreover, if $X$ is projective space, a Grassmannian variety, or a smooth complete intersection in $\PP^n$ of dimension at least three, then $\Cl(X)=\ZZ$. It follows that, in these cases, the divisor class group of $R(X,D)$ is finite for any ample $\QQ$-divisor $D$ on $X$, and hence that $R(X,D)$ is \emph{almost factorial} in the sense of Storch \cite{Storch}.

Lipman proved that the divisor class group of a two-dimensional normal local ring $R$ with rational singularities is finite, \cite[Theorem~17.4]{Lipman_IHES}. While this is a hard result, the analogous statement for graded rings is a straightforward application of Theorem~\ref{exactseq}. Indeed, let $R$ be an $\NN$-graded normal ring of dimension two, finitely generated over an algebraically closed field $R_0$, such that $R$ has rational singularities. Then $R$ has a negative $a$-invariant by \cite[Theorem~3.3]{Watanabe}, so $H^1(X,\calO_X)=0$ where $X=\Proj R$. But then $X$ is a curve of genus $0$ so it must be $\PP^1$, and it follows that the divisor class group of $R$ is finite.

\begin{remark}\label{rem:KWproof}
We note some aspects of Watanabe's proof of Theorem~\ref{exactseq}. Let $\Div(X)$ be the group of Weil divisors on $X$, and let
\[
\Div(X,\QQ)=\Div(X)\otimes_\ZZ\QQ
\]
be the group of $\QQ$-divisors. For $D$ as in Theorem~\ref{exactseq}, set $\Div(X,D)$ to be the subgroup of $\Div(X,\QQ)$ generated by $\Div(X)$ and the divisors
\[
\frac{1}{q_1}V_1,\ \dots,\ \frac{1}{q_r}V_r\,.
\]
Each element $E\in\Div(X,D)$ gives a divisorial ideal
\[
\oplus_{j\ge0}H^0(X,\calO_X(E+jD))T^j
\]
of $R$, and hence an element of $\Cl(R)$. The map $\Div(X,D)\to\Cl(R)$ induces
a surjective homomorphism
\[
\Div(X,D)/\Div(X)\to \Cl(R)/\image(\Cl(X))\,.
\]
\end{remark}

\section{Computing divisor class groups}

The divisor class groups of affine surfaces of characteristic $p$ defined by equations of the form $z^{p^n}=g(x,y)$ have been studied in considerable detail; such surfaces are sometimes called \emph{Zariski surfaces}. In \cite{LangTams} Lang computed the divisor class group of hypersurfaces of the form $z^{p^n}=g(x_1,\dots,x_d)$ where $g$ is a homogeneous polynomial of degree relatively prime to $p$. The proposition below recovers \cite[Proposition~3.11]{LangTams}.

Let $A=k[x_1,\dots,x_d]$ be a polynomial ring over a field. We say $g\in A$ is a \emph{weighted homogeneous polynomial} if there exists an $\NN$-grading on $A$, with $A_0=k$, for which $g$ is a homogeneous element.

\begin{proposition}\label{prop:main}
Let $R=k[z,x_1,\dots,x_d]/(z^n-g)$ be a normal hypersurface over a field $k$, where $g\in k[x_1,\dots,x_d]$ is a weighted homogeneous polynomial with degree relatively prime to $n$. Let $g=h_1\cdots h_r$, where $h_i\in k[x_1,\dots,x_d]$ are irreducible polynomials. Then
\[
\Cl(R)=(\ZZ/n\ZZ)^{r-1}\,,
\]
and the images of $(z,h_1),\dots,(z,h_{r-1})$ form a minimal generating set for $\Cl(R)$.
\end{proposition}

Note that if $n\ge2$, then the hypothesis that $R$ is normal forces $h_1,\dots,h_r$ to be pairwise coprime irreducible polynomials. 

\begin{proof}[Proof of Proposition~\ref{prop:main}]
The polynomial ring $k[x_1,\dots,x_d]$ has a grading under which $\deg x_i=c_i$ for $c_i\in\NN$, and the degree of $g$ is an integer $m$ relatively prime to $n$. We assume, without any loss of generality, that $\gcd(c_1,\dots, c_d)=1$. Consider the $\NN$-grading on $R$ where $\deg x_i=n c_i$ and $\deg z=m$. Note that under this grading $\deg g=\sum\deg h_i=mn$. The $n$-th Veronese subring of $R$ is
\[
R^{(n)}=k[z^n,x_1,\dots,x_d]/(z^n-g)=k[x_1,\dots,x_d]\,,
\]
which is a polynomial ring in $x_1,\dots,x_d$. Let $X=\Proj R^{(n)}=\Proj R$.

There exist integers $s_i$, $a$, and $b$ such that $\sum_{i=1}^ds_i c_i=1$ and $am+bn=1$. Consider the $\QQ$-divisor on $X$ given by
\[
D=b\div(\bx)+\frac{a}{n}\div(g)=b\sum_{i=1}^ds_iV(x_i)+\frac{a}{n}\sum_{i=1}^rV(h_i)\,,
\]
where $\bx=x_1^{s_1}\cdots x_d^{s_d}$. We claim that 
\begin{equation}\label{correct_d}
R=\oplus_{j\ge0}H^0(X,\calO_X(jD))T^j\,,
\end{equation}
where $T=z^a\bx^b$ is a homogeneous degree $1$ element of the fraction field of $R$. First note that $\lf am/n\rf = \lf(1-bn)/n\rf =-b$, so
\[
\lf mD\rf=bm\div(\bx)+\left\lf\frac{am}{n}\right\rf\div(g)=bm\div(\bx)-b\div(g)\,.
\]
Consequently $\deg\lf mD\rf=0$, and $H^0(X,\calO_X(mD))T^m$ is the $k$-vector space spanned by the element
\[
\bx^{-bm}g^bT^m=\bx^{-bm}(z^n)^b(z^a\bx^b)^m=z^{bn+am}=z\,.
\]
Let $c=c_t$ for an integer $1\le t\le d$. Then $ncD=bnc\div(\bx)+ac\div(g)$ has degree $nc$, and $H^0(X,\calO_X(ncD))T^{nc}$ contains the element
\[
x_t\bx^{-bnc}g^{-ac}T^{nc}=x_t\bx^{-bnc}(z^n)^{-ac}(z^a\bx^b)^{nc}=x_t\,.
\]
To prove the claim \eqref{correct_d}, it remains to verify that $z,x_1,\dots,x_d$ are $k$-algebra generators for the ring $\oplus_{j\ge0}H^0(X,\calO_X(jD))T^j$. An arbitrary positive integer $j$ can be written as $um+vn$ for $0\le u\le n-1$. We then have
\begin{align*}
\lf jD\rf&=b(um+vn)\div(\bx)+\left\lf\frac{a(um+vn)}{n}\right\rf\div(g)\\
&=b(um+vn)\div(\bx)+(va-ub)\div(g)\,,
\end{align*}
which has degree $vn$. Consequently $H^0(X,\calO_X(jD))T^j$ vanishes if $v$ is negative, and for nonnegative $v$, it is spanned by elements
\[
\mu\bx^{-b(um+vn)}g^{-va+ub} T^{um+vn}=\mu z^u\,,
\]
for monomials $\mu$ in $x_i$ of degree $v$. This completes the proof of \eqref{correct_d}.

Since $nD$ has integer coefficients, the exact sequence of Theorem~\ref{exactseq} for the divisor $nD$ and corresponding ring $R^{(n)}$ reduces to
\[\CD
0@>>>\ZZ@>{\theta}>>\Cl(X)@>>>\Cl(R^{(n)})@>>>0\,,
\endCD\]
where $\theta(1)=nD$. Since $R^{(n)}$ is a polynomial ring, and hence factorial, it follows that $nD$ generates $\Cl(X)$. Next, consider the exact sequence applied to the divisor $D$ and corresponding ring $R$, i.e., the sequence
\[\CD
0@>>>\ZZ@>{\theta}>>\Cl(X)@>>>\Cl(R)@>>>\coker\alpha@>>>0\,.
\endCD\]
The lcm of the denominators occurring in $D$ is $n$, so we once again have $\theta(1)=nD$. Consequently $\theta$ is an isomorphism and $\Cl(R)=\coker\alpha$, where
\[
\alpha\colon\ZZ\to\bigoplus_1^r\ZZ/n\ZZ\qquad\text{ with }\qquad\alpha(1)=(a,\dots,a)\,.
\]
Since $a$ and $n$ are relatively prime, it follows that
\[
\Cl(R)=(\ZZ/n\ZZ)^{r-1}\,.
\]

We next determine explicit generators for $\Cl(R)$ by Remark~\ref{rem:KWproof}. The $\QQ$-divisors
\[
E_t=-\frac{1}{n}V(h_t)\qquad\text{ for }1\le t\le r
\]
give a generating set for $\Div(X,D)/\Div(X)$ which surjects onto $\Cl(R)$. Hence the divisorial ideals
\[
\frakp_t=\oplus_{j\ge0}H^0(X,\calO_X(E_t+jD))T^j\qquad\text{ where }1\le t\le d\,,
\]
generate $\Cl(R)$. The computation of $\frakp_t$ is straightforward, and we give a brief sketch. First note that 
\begin{align*}
\lf E_t+mD\rf&=bm\div(x)+\left\lf\frac{am-1}{n}\right\rf V(h_t)+\sum_{i\neq t}\left\lf\frac{am}{n}\right\rf V(h_i)\\
&=bm\div(\bx)-b\div(g)\,,
\end{align*}
so $H^0(X,\calO_X(E_t+mD))T^m$ is the $k$-vector space spanned by
\[
\bx^{-bm}g^bT^m=z.
\]
Since the degree of each $x_i$ is a multiple of $n$, we have $\deg h_t=n\gamma$ for some integer $\gamma$. We next compute the component of $\frakp_t$ in degree $n\gamma$. Note that
\[
\lf E_t+n\gamma D\rf=-V(h_t)+bn\gamma\div(\bx)+a\gamma\div(g)\,,
\]
so $H^0(X,\calO_X(E_t+n\gamma D))T^{n\gamma}$ is the $k$-vector space spanned by
\[
h_t\bx^{-bn\gamma}g^{-a\gamma}T^{n\gamma}=h_t.
\]
It is now a routine verification that $z,h_t$ are generators for the ideal $\frakp_t$, which, we note, is a height one prime of $R$. Consequently $\Cl(R)$ is generated by $\frakp_1,\dots,\frakp_r$. Using $\sim$ to denote linear equivalence, we have
\[
nE_t+n\gamma D\sim 0\qquad\text{ and }\qquad\sum_{i=1}^r E_i+mD\sim 0\,,
\]
implying that $n[\frakp_t]=0$ and $\sum_i[\frakp_i]=0$ in $\Cl(R)$. These correspond to the calculations with divisorial ideals,
\[
\frakp_t^{(n)}=h_tR\qquad\text{ and }\qquad\bigcap_{i=1}^r\frakp_i=zR\,,
\]
and imply, in particular, that $[\frakp_1],\dots,[\frakp_{r-1}]$ is a generating set for $\Cl(R)$.
\end{proof}

\begin{example}
We use Proposition~\ref{prop:main} to compute the divisor class group of diagonal hypersurfaces
\[
R=k[z,x_1,\dots,x_d]/(z^n-x_1^{m_1}-\cdots-x_d^{m_d})
\]
where $n$ is relatively prime to $m_i$ for $1\le i\le d$, and $k$ is a field of characteristic zero, or of characteristic not dividing each $m_i$.

By the Jacobian criterion, $R$ has an isolated singularity at the homogeneous maximal ideal $\frakm$. Hence if $d\ge 4$, then $R$, as well as its $\frakm$-adic completion $\widehat{R}$, are factorial by Grothendieck's parafactoriality theorem \cite{Grothendieck}; see \cite{Call_Lyubeznik} for a simple proof of Grothendieck's theorem.

\medskip

\emph{Case $d=3$.} The polynomial $g=x_1^{m_1}+x_2^{m_2}+x_3^{m_3}$ is irreducible since $k[x_1,x_2,x_3]/(g)$ is a normal domain by the Jacobian criterion. We set $\deg x_i$ to be $m_1m_2m_3/m_i$. Then $g$ is a weighted homogeneous polynomial of degree $m_1m_2m_3$, which is relatively prime to $n$, so Proposition~\ref{prop:main} implies that $R$ is factorial. Since $R$ satisfies the Serre conditions $(R_2)$ and $(S_3)$, the completion $\widehat{R}$ is factorial as well by \cite[Korollar~1.5]{Flenner}. The divisor class groups of rational three-dimensional Brieskorn singularities are computed in \cite[Chapter~IV]{BS}; see also \cite{Storch_Crelle}.

\medskip

\emph{Case $d=2$.} Let $g=x_1^{m_1}+x_2^{m_2}$. If $c=\gcd(m_1,m_2)$, let $m_1=ac$ and $m_2=bc$, and set $\deg x_1=b$ and $\deg x_2=a$. Let $f$ be an irreducible factor of $g$. Then $f$ is homogeneous, and hence has the form $\sum a_{ij}x_1^ix_2^j$ where $a_{ij}\in k$ and $bi+cj=\deg f$ for each term occurring in the summation. Since $x_1$ and $x_2$ do not divide $g$, we see that $f$ must contain nonzero terms of the form $a_{0j}x_2^j$ and $a_{i0}x_1^i$. Hence $\deg f$ is a multiple of $ab$, and it follows that $f$ is a polynomial in $x_1^a$ and $x_2^b$. Consequently the number of factors of $g$ in $k[x_1,x_2]$ is the number of factors of $s^c+t^c$ in $k[s,t]$ or, equivalently, the number of factors of $1+t^c$ in $k[t]$.

In particular, if $m_1$ and $m_2$ are relatively prime, then $g$ is irreducible and Proposition~\ref{prop:main} implies that $R$ is factorial. As is well-known, $\widehat{R}$ need not be factorial; see for example, \cite[Theorem~III.5.2]{Samuel}.

If $k$ is algebraically closed, then $g$ is a product of $c$ irreducible factors, and so Proposition~\ref{prop:main} implies that
\[
\Cl(R)=(\ZZ/n\ZZ)^{c-1}\,.
\]
\end{example}

\begin{remark}
The condition that the degree of $g$ is relatively prime to $n$ is certainly crucial in Proposition~\ref{prop:main}. In the absence of this, $\Cl(R)$ need not be finite, for example $\CC[z,x_1,x_2,x_3]/(z^3-x_1^3-x_2^3-x_3^3)$ has divisor class group $\ZZ^6$. However, one can drop the relatively prime condition when considering hypersurfaces of the form $z^n-x_0g(x_1,\dots,x_d)$, see also \cite[Proposition~3.12]{LangTams}:
\end{remark}

\begin{corollary}\label{cor:x0}
Let $R=k[z,x_0,\dots,x_d]/(z^n-x_0g)$ be a normal hypersurface over a field $k$, where $g$ is a weighted homogeneous polynomial in $x_1,\dots,x_d$. Let $g=h_1\cdots h_r$, where $h_i\in k[x_1,\dots,x_d]$ are irreducible. Then
\[
\Cl(R)=(\ZZ/n\ZZ)^r\,,
\]
and the images of $(z,h_1),\dots,(z,h_r)$ form a minimal generating set for $\Cl(R)$.
\end{corollary}

\begin{proof}
We may choose the degree of $x_0$ such that $\deg(x_0g)$ is relatively prime to $n$. The result then follows from Proposition~\ref{prop:main}.
\end{proof}

We conclude with the following example.

\begin{example}
Let $k$ be a field. Corollary~\ref{cor:x0} implies that the divisor class group of the ring $R=k[xy,x^n,y^n]$ is $\ZZ/n\ZZ$, since $R$ is isomorphic to the hypersurface 
\[
k[z,x_0,x_1]/(z^n-x_0x_1)\,.
\]
In \cite[Chapter~III]{Samuel}, the divisor class group of $R$ is computed by Galois descent if $n$ is relatively prime to the characteristic of $k$, and by using derivations if $n$ equals the characteristic of $k$.
\end{example}


\end{document}